\documentclass[11pt,reqno]{amsart}
\usepackage{amssymb,amscd,amsbsy,mathrsfs}
\usepackage{mathdots}
\renewcommand{\a}{\alpha}
\renewcommand{\b}{\beta}

\newcommand{\e}{\varepsilon}
\newcommand{\vk}{\varkappa}
\newcommand{\z}{\zeta}

\newcommand{\vt}{\vartheta}
\renewcommand{\l}{\lambda}

\newcommand{\s}{\sigma}

\newcommand{\f}{\varphi}

\newcommand{\h}{{\mathscr H}}

\newcommand{\C}{{\Bbb C}}
\newcommand{\T}{{\Bbb T}}

\newcommand{\dd}{{\Bbb D}}
\newcommand{\R}{{\Bbb R}}
\newcommand{\Z}{{\Bbb Z}}

\newcommand{\0}{{\boldsymbol{0}}}

\newcommand{\bs}{\boldsymbol}

\newcommand{\m}{{\boldsymbol m}}
\newcommand{\bS}{{\boldsymbol S}}

\newcommand{\rf}[1]{(\ref{#1})}

\newcommand{\df}{\stackrel{\mathrm{def}}{=}}

\newcommand{\Ker}{\operatorname{Ker}}
\newcommand{\re}{\operatorname{Re}}

\newcommand{\supp}{\operatorname{supp}}
\newcommand{\clos}{\operatorname{clos}}

\newcommand{\trace}{\operatorname{trace}}

\newcommand{\const}{\operatorname{const}}
\newcommand{\tr}{\operatorname{trace}}
\newcommand{\eeq}{\end{equation}}
\newcommand{\beq}{\begin{equation}}
\newcommand{\bay}{\begin{eqnarray}}
\newcommand{\ba}{\begin{align*}}
\newcommand{\ea}{\end{align*}}
\newcommand{\ey}{\end{eqnarray}}
\newcommand{\bey}{\begin{eqnarray*}}
\newcommand{\eey}{\end{eqnarray*}}

\newcommand{\be}{\infty}

\newcommand{\bl}{\blacksquare}

\newcommand{\Range}{\operatorname{Range}}

\newcommand{\Pf}{{\bf Proof. }}
\newcommand{\im}{\operatorname{Im}}
\renewcommand{\re}{\operatorname{Re}}

\newtheorem{thm}{\hspace{\parindent}Theorem}[section]

\newtheorem{cor}[thm]{\hspace{\parindent}Corollary}
\newtheorem{lem}[thm]{\hspace{\parindent}Lemma}

\usepackage{color}

\theoremstyle{remark}

\newtheorem*{rem*}{Remark}

\newcommand\CA{{\rm C}_{\rm A}}
\newcommand{\OL}{{\rm OL}}

\newcommand{\ri}{{\rm i}}
\newcommand{\ran}{\rm ran}
\newcommand{\dom}{{\rm dom}}

\newcommand{\rd}{{\rm d}}

\begin{document}

\newcommand{\vse}{\vspace{.2in}}
\numberwithin{equation}{section}

\title{Real spectral shift functions for pairs of contractions and pairs of dissipative operators}
\author{M.M. Malamud, H. Neidhardt and V.V. Peller}
\thanks{The research on \S\:3-5  is supported by
Russian Science Foundation [grant number 23-11-00153]. }

%\author
%\thanks{Corresponding author: V.V. Peller; email: peller@math.msu.edu}

\maketitle

\hfill to the blessed memory of Heinz Langer

\

\begin{abstract}
Recently the authors
solved a long-standing problem and showed that for an arbitrary pair of contractions on Hilbert space with trace class difference has an integrable spectral shift function on the unit circle $\T$ and an analogue of the Lifshits--Krein trace formula holds. It is also known that it may happen that there is no real-values integrable spectral shift function.
In this paper we find conditions under which a pair of contractions with trace class difference has {\it a real-valued integrable} spectral shift function. We also consider a similar problem for pairs of dissipative operators.
Finally, we find an application of the results in question to dissipative Schr\"odinger operators.
\end{abstract}

%\maketitle

%\footnotesize
%\tableofcontents
%\normalsize
%\vspace*{-1cm}

\setcounter{section}{0}
\section{\bf Introduction}
\setcounter{equation}{0}
\label{In}

\medskip

The purpose of this paper is to find conditions, under which a pair of contractions  with trace class difference possesses a real-valued integrable spectral shift function. We also study the problem to find conditions for the existence of a real-valued spectral shift function for a pair of dissipative operators with trace class difference.

It was physicist I.M. Lifshits who introduced in \cite{L} the notion of spectral shift function for pairs of self-adjoint operators with trace class difference and offered a trace formula. In that paper he associated with a pair of self-adjoint operators $A_0$ and $A_1$, whose difference is of trace class\footnote{We refer the reader to the book \cite{GK} for an introduction to the trace class $\bS_1$ and Schatten--von Neumann classes $\bS_p$.}, a real-valued integrable function
$\bs{\xi}=\bs{\xi}_{A_0,A_1}$ called the {\it spectral shift function for the pair $\{A_0,A_1\}$} and offered the trace formula
\bay
\label{foslsso}
\trace\big(f(A_1)-f(A_0)\big)=\int_\R f'(t)\bs{\xi}_{A_0,A_1}(t)\,\rd t
\ey
for sufficiently nice functions $f$ on the real line $\R$. It turned out, however, that mathematically the reasonings in \cite{L} were not mathematically rigorous.

Later M.G. Krein in \cite{Kr} offered a mathematically rigorous approach and considered the most general situation. He proved that for an arbitrary pair of self-adjoint operators $\{A_0,A_1\}$ with trace class difference that exists a {\it unique} function $\bs{\xi}=\bs{\xi}_{A_0,A_1}$ in $L^1(\R)$ such that trace formula \rf{foslsso}
holds for sufficiently nice functions $f$. Moreover, $\bs\xi$ must be {\it real-valued}.
We are going to call formula \rf{foslsso} the {\it Lifshits--Krein trace formula}.

In the same paper \cite{Kr} Krein showed that
 $$
 \bs{\xi}_{A_1,A_0}(t) = \frac{1}{\pi}\lim_{y\downarrow 0}\im (\log(\Delta_{A_1/A_0}(t +\ri y)))
\quad \mbox{for a.e.} \quad  t \in \R,
$$
where  the {\it perturbation determinant} $\Delta_{A_1/A_0}$ is defined by
$$
\Delta_{A_1/A_0}(\z)\df\det(I + (A_1-A_0)(A_0-\z I)^{-1}),\quad\z\not\in\R,
$$
(see also \cite{Ya}, Ch. 8 and \cite{Sch}, Ch. 9; we also mentioned here the surveys \cite{BP} and \cite{BY} on spectral shift function).

Later Krein extended in \cite{K64} formula \eqref{foslsso} to the class $\mathcal W_1(\R)$  of functions whose derivative is the Fourier transform of a complex Borel measure.

Krein also observed in \cite{Kr} that the right-hand side of \rf{foslsso} makes sense for arbitrary Lipschitz functions $f$ and posed the problem to describe the class of functions, for which formula \rf{foslsso} holds for all pairs of self-adjoint operators with difference of trace class $\bS_1$.

However, Farforovskaya proved in \cite{F} that trace formula \rf{foslsso} cannot be generalized to the class of all Lipschitz functions. Indeed, it was shown in \cite{F} that there exist a Lipschitz function $f$ on $\R$ and self-adjoint operators $A_1$ and $A_0$ such that $A_1-A_0\in\bS_1$, but $f(A_1)-f(A_0)\not\in\bS_1$.

Later it was shown in \cite{Pe1} and \cite{Pe3} that \rf{foslsso} holds for functions $f$ in the (homogeneous) Besov space $B_{\be,1}^1(\R)$. It was also shown in \cite{Pe1} and \cite{Pe3} that if trace formula \rf{foslsso} holds for arbitrary pairs of self-adjoint operators with trace class difference, then $f$ locally belongs to the Besov space $B_{1,1}^1$.

Krein's problem was completely resolved in \cite{Pe6} where it was shown that trace formula \rf{foslsso} holds for arbitrary pairs $(A_0,A_1)$ of not necessarily bounded self-adjoint operators with trace class difference if and only if $f$ is an {\it operator Lipschitz function}\footnote{We refer the reader to the survey \cite{AP} for detailed information on operator Lipschitz functions}, i.e., the inequality $$
\|f(A)-f(B)\|\le\const\|A-B\|
$$
holds for arbitrary self-adjoint operators $A$ and $B$.

Further development of Lifshits--Krein formula \rf{foslsso} for pairs of self-adjoint operators was
proposed by Alpay and Gohberg in \cite{AG1} and \cite{AG2}. They proposed different formulae for the spectral
shift function (see also the paper [MN] in this connection). They investigated in details the
case of finite rank perturbations.

In \cite{Kr2} Krein introduced the notion of spectral shift function for pairs of unitary operators with trace class difference. Namely, for a pair of unitary operators $(U_1,U_0)$ with trace class difference, he proved that there exists a {\it real-valued} function $\bs{\xi}$ in $L^1(\T)$, unique modulo an additive constant, (called a {\it spectral shift function for} $(U_0,U_1)$ ) such that the trace formula
\bay
\label{fosuo}
\trace\big(f(U_1)-f(U_0)\big)=\int_\T f'(\z)\bs{\xi}(\z)\,\rd\z
\ey
holds for functions $f$ whose derivative $f'$ has absolutely convergent Fourier series.

An analog of the result of \cite{Pe6} for unitary operators was obtained in
 \cite{AP+}. Namely, it was shown in \cite{AP+} that
the class of functions $f$, for which formula \rf{fosuo} holds for arbitrary pairs $\{U_1,U_0\}$ of unitary operators with $U_1-U_0\in\bS_1$ coincides with the class $\OL_\T$ of operator Lipschitz functions on $\T$.

The problem to find analogs of the above trace formulae for functions of contractions and maximal dissipative operators has been under consideration since long ago.

Recall that an operator $\T$ on Hilbert space is called a {\it contraction} if $\|T\|\le1$.
A densely defined operator $L$ is called {\it dissipative} if $\im(Lx,x)\ge0$ for an arbitrary vector $x$ in the domain ${\rm dom}(L)$ of the operator $L$. A dissipative operator $L$ is called maximal dissipative if it has no proper dissipative extension.

%In this paper we consider the case of {\it contractions}, i.e., operators $T$ on Hilbert space such that $\|T\|\le1$.

Recall that the Sz.-Nagy--Foia\c s functional calculus (see \cite{SNF})
associates with each function $f$ in the {\it disk-algebra} $\CA$ (i.e., the space of functions analytic in the disk $\dd$ and continuous in its closure) for a contraction $T$ on Hilbert space establishes the linear and multiplicative functional calculus
$$
f\mapsto f(T),\quad f\in\CA,
$$
that satisfies the von Neumann inequality
$$
\|f(T)\|\le\|f\|_{\CA}\df\max\{|f(\z)|:~|\z|\le1\},\quad f\in\CA.
$$

There were many attempts to generalize the Lifshits--Krein trace formula to the case of functions of contractions. In other words, the problem is for a pair of contractions $(T_0,T_1)$ with trace class difference, to establish the existence of an integrable function $\bs\xi%=\bs\xi_{\{T_1,T_0\}}
$
on $\T$ such that
\bay
\label{forsleszha}
\trace\big(f(T_1)-f(T_0)\big)=\int_\T f'(\z)\bs{\xi}(\z)\,\rd\z
\ey
for sufficiently nice functions $f$. Another related problem was to describe the maximal class of functions
$f$, for which trace formula \rf{forsleszha} holds for arbitrary pairs $(T_0,T_1)$ of contractions with trace class difference.

Similar problems for pairs of maximal dissipative operators also remained open for a long time.

The first result obtained for nonself-ajoint and nonunitary operators was obtained by Heinz Langer \cite{La} who considered the case of contractions whose spectra are subsets of the open unit disc $\dd$ that includes the case of strict contractions
and proved in this case the existence of an integrable spectral shift function. Actually, he considered in \cite{La} even a more general situation that involves operators that are not necessarily contractions.

We would also like to mention partial results by Rybkin \cite{Ryb84,Ryb87,Ryb89,Ryb94},
 Adamyan  and Neidhardt \cite{AN} and M.G. Krein \cite{Kr87}. Note here that in his papers \cite{Ryb87,Ryb89,Ryb94} Rybkin under additional assumptions established the existence of
 a complex-valued spectral shift function that is A-integrable but not necessarily
Lebesgue integrable.

We also mention here that in \cite{AN} the existence of a real-valued integrable spectral shift function
for a pair of contractions $(T_0,T_1)$ was established under the stronger assumption than $T_1-T_0\in\bS_1$:
\bay
\label{logarifmy}
\sum_{k\ge0}s_k(T_1-T_0)\log\big(1+(s_k(T_1-T_0))^{-1}\big)<\be
\ey
(it is assumed that the function $x\mapsto x\log(1+x^{-1})$ takes value 0 for $x=0)$.
It follows easily from our Theorem \ref{veshch} in \S\;2  that condition \rf{logarifmy} is not necessary for the existence of a real-valued integrable spectral shift function.

The longstanding problem to obtain an analogue of the Lifshits--Krein trace formula for functions of contractions was solved completely in \cite{MNP1} (see the earlier paper \cite{MalNei2015}, in which such a trace formula was obtained under an additional assumption) and in \cite{MNP2} by using completely different approaches.
In \cite{MNP1} Krein's approach based on perturbation determinants was used while the paper \cite{MNP2}
develops an idea of Birman and Solomyak \cite{BS1} that uses double operator integrals.

Note also that in \cite{BS1} the authors did not obtain the existence of an integrable spectral shift function for pairs of self-adjoint operators with trace class difference. They rather established the existence of a spectral shift measure and their method did not give the absolute continuity of that measure. The crucial point in \cite{MNP2} to obtain the absolute continuity of the spectral shift measure was an application of the Sz.-Nagy--Foia\c s theorem that says that the minimal unitary dilation of a completely nonunitary contraction has absolutely continuous spectral measure.

Note that unlike the case of pairs of self-adjoint operators with trace class difference a spectral shift function for pairs of contractions with trace class difference is never unique. If $\bs\xi$ is a spectral shift function for a pair of  contractions with trace class difference, then all spectral shift functions for the same pair admits parametrization
$\{\bs{\xi}+h\}$, where $h$ ranges over the Hardy class $H^1$.

A question naturally arises of whether it is possible to select a real-valued spectral shift function. In general the answer is negative: there exist pairs $(T_0,T_1)$ of contractions with trace class difference, for which
there is no real-valued integrable spectral shift function, see \cite{MNP1}. However, as shown in \cite{MalNei2015}, in the case of a finite rank perturbation a real-valued integrable function always exists.

On the other hand, it was established in \cite{MNP2} that for an arbitrary pair $\{T_0,T_1\}$ of contractions with trace class difference, there exists a {\it real-valued} A-integrable spectral shift function in which case
in formula \rf{forsleszha} the right-hand side should be understood in the sense of A-integral.

Finally, let us mention here that the maximal class of functions $f$ for which trace formula \rf{forsleszha}
holds for arbitrary pairs $(T_0,T_1)$ of contractions with trace class difference coincides with the class of operator Lipschitz functions analytic in $\dd$, see \cite{MNP2}.

The Lifshits--Krein trace formula was also generalized in \cite{MalNei2015}, \cite{MNP1} and \cite{MNP2} to the case of pairs of maximal dissipative operators. If $L_0$ and $L_1$ are maximal dissipative operators with trace class difference, then it was shown in \cite{MNP1} and \cite{MNP2} that there exists an integrable spectral shift function $\bs\xi$ on the real line $\R$ such that the trace formula
\bay
\label{fcddo}
\trace\big(f(L_1)-f(L_0)\big)=\int_\R f'(t)\bs{\xi}(t)\,\rd t
\ey
for rational functions $f$ with poles in the open lower half-plane.

On the other hand, if instead of the condition $L_1-L_0\in\bS_1$ we impose the resolvent condition
$$
(\ri I+L_1)^{-1}-(\ri I+L_0)^{-1}\in\bS_1,
$$
there exists a spectral shift function $\bs\xi$ on $\R$ such that
\bay
\label{s_vesom}
\int_\R\frac{|\bs{\xi}(t)|}{1+t^2}\,\rd t<\be
\ey
and trace formula \rf{fcddo} holds for rational functions $f$ with poles in the open lower half-plane,
see \cite{MNP2}.

In this paper we are going to find conditions under which a pair of contractions with trace class difference has a {\it real-valued} integrable spectral shift function. Recently Chattopadhyay and Sinha showed in \cite{CS}
that if $T_0$ and $T_1$ are contractions with trace class difference and $T_0$ is a strict contraction, i.e.,
$\|T_0\|<1$, then the pair $\{T_0,T_1\}$ possesses a real-valued integrable spectral shift function.

We are going to essentially improve this result in \S\:\ref{szhali} and prove the existence of a real-valued integrable spectral shift function under much more relaxed assumptions.

It is also appropriate to mention here that it was shown in \cite{MNP2} that if $T$ is a contraction and $U$ is a unitary operator such that $U-I\in\bS_1$, then the pair $(T,UT)$ has a real-valued integrable spectral shift function (see Lemma 9.1 in \cite{MNP2}). On the other hand, if $T$ is a contraction and $X$ is a contraction such that $X\ge\0$ and $I-X\in\bS_1$, then the pair $(T,XT)$ has a purely imaginary integrable spectral shuft function (see Lemma 9.2 in \cite{MNP2}).

In \S\:\ref{dissipatsiya} of this paper we are going to work on conditions for the existence of a real-valued spectral shift function for pairs $(L_0,L_1)$ of maximal dissipative operators with trace class resolvent difference.  We find new sufficient conditions under which the pair $(L_0,L_1)$ has a real-valued spectral shift function $\bs\xi$ satisfying \rf{s_vesom} and complement certain results in this direction obtained in \cite{MalNei2015},
Sec. 5 and \cite{MNP2}, Sec. 9. 
However, it turns out that under our condition, the pair $(L_0,L_1)$ cannot possess a real-valued {\it integrable} spectral shift function unless $\trace(L_1-L_0)\in\R$.

In this paper we improve the results of our note \cite{MNP3} where certain preliminary results were announced.

We devote this paper to  Heinz Langer remarkable mathematician and personality. The first author highly appreciates his support for several years of his mathematical life.

\

\section{\bf Real-valued integrable spectral shift functions for pairs of contractions}
\setcounter{equation}{0}
\label{szhali}

\

In this section we obtain a condition on a pair  $(T_0,T_1)$ of contractions with trace class difference, under which this pair has a real-valued integrable spectral shift function. To establish the main result of this section, we prove Theorem \ref{Fyokla}. Obviously, if we know that $X$ and $Y$ are operators on Hilbert space satisfying the conditions $\0 \le X\le I$ and $\0 \le Y\le I$ and $Y - X \in \bS_1$, then for $\s\in(0,1)$, the operator
$Y^\s - X^\s$ does not have to belong to $\bS_1$. However, Theorem \ref{Fyokla} shows that if we impose a certain additional assumption, we can make the conclusion that $Y^\s - X^\s\in\bS_1$.

We are going to introduce the following piece of notation. Suppose that $Q$ is a densely defined linear operator on Hilbert space that extends by continuity. We are going to use the the same notation $Q$ for the unique extension of $Q$ to a bounded linear operator.

\begin{thm}
\label{Fyokla}
Let $p>0$ and $\s\in (0,1)$.
Suppose that $X$ and $Y$ are operators on Hilbert space such that $\0 \le X\le I$, $\0 \le Y\le I$,
$\Ker X = \{\0\}$ and $\Ker Y = \{\0\}$. If $Y - X \in \bS_p$ and the operator
$Y^{-\b}(Y-X)X^{-\a}$ is densely defined and extends to an operator  of class $\bS_p$ for some nonnegative
$\a$ and $\b$ such that
$\a+\b\in(1-\s,1]$, then $Y^{\s} - X^{\s} \in \bS_p$.
\end{thm}

\Pf
We use the notation $\clos(Y^{-\b}(Y-X)X^{-\a})$ for the closure of the operator $Y^{-\b}(Y-X)X^{-\a}$, which belongs to $\bS_p$ by the assumption.
%%It is easy to deduce from
The following elementary identity is well known  (see e.g., \cite{BS2}, Ch. 10, Sect. 4).
$$
   c_\s \int_0^\be\frac{t^{\s-1}}{t+s}\,\rd t = s^{\s - 1}, \quad c_\s = \frac{\sin\pi\s}{\pi}, \quad s>0,\quad  0< \s <1,
$$
It follows easily and it is well known (see e.g. Bir-Sol, Ch.10) that
$$
Y^{\s} = c_\s\int^\infty_0t^{\s - 1}Y(tI + Y)^{-1}\rd t
\quad\mbox{and} \quad
X^{\s} =c_\s\int^\infty_0t^{\s - 1}X(tI+X)^{-1}\rd t.
$$
It is easy to verify that
$$
(tI + Y)\big(Y(tI + Y)^{-1}-X(tI+X)^{-1}\big)(tI+X)=t(X-Y),
$$
and so
$$
Y(tI + Y)^{-1}-X(tI+X)^{-1}=t(tI + Y)^{-1}(X-Y)(tI+X)^{-1}.
$$
Thus,
\begin{align*}
Y^{\s}-X^{\s}& = c_\s\int^\infty_0t^{\s}(tI + Y)^{-1}(X-Y)(tI+X)^{-1}\,\rd t\\[.2cm]
&=  c_\s\int^1_0t^{\s}\big((tI + Y)^{-1}Y^{\b}\big) \clos\big(Y^{-\b}(Y-X)X^{-\a}\big)\big(X^\a(tI+X)^{-1}\big)
\,\rd t\\[.2cm]
& + c_\s\int^\infty_1t^{\s}(tI + Y)^{-1}(X-Y)(tI+X)^{-1}\,\rd t.
\end{align*}
Since $Y$ is non-negative, it is easy to see that
\bay
\label{pervoe}
\left\|t^{\s}(tI+Y)^{-1}\right\| \le t^{\s -1},  \qquad t\in (0,\infty).
\ey
%and
%\bay
%\label{vtoroe}
% \left\|t^{1/2}(tI+Y)^{-1}\right\| \le t^{\alpha-1}, \quad t\in [1,\infty),
%\ey
Clearly,
\begin{align}
\label{eq:first_est-te_for_Y1/2-X1/2}
\!\!\!\big\|Y^{\s} - X^{\s}\big\|_{\bS_p}
&\le  c_\s \left\|\int^1_0t^{\s}\big((tI + Y)^{-1}Y^{\b}\big)\big(Y^{-\b}(Y-X)X^{-\a}\big)\big(X^\a(tI+X)^{-1}\big)\,\rd t
\right\|_{\bS_p}\nonumber\\[.2cm]
& + c_\s\left\| \int^\infty_1t^{\s}(tI + Y)^{-1}(X-Y)(tI+X)^{-1}\,\rd t\right\|_{\bS_p}.
\end{align}

Clearly, the first term on the right of \rf{eq:first_est-te_for_Y1/2-X1/2} can be estimated
in terms of 
\begin{align}
\label{eq:first_main_estimae}
&  c_\s\left\|\int^1_0t^{\s} \big((tI + Y)^{-1}Y^{\b} \big) \big(Y^{-\b}(Y-X)X^{-\a}\big)
\big(X^\a(tI+X)^{-1}\big)\,\rd t\right\|_{\bS_p}\nonumber\\[.2cm]
&\hspace*{2cm}
\le
c_\s \int^1_0  \left\|(tI+Y)^{-1}Y^\b\right\| \left\|t^{\s}Y^{-\b}(Y-X)X^{-\a}\right\|_{\bS_p}
\left\|X^\a(tI+X)^{-1}\right\|\,\rd t.
\end{align}

 The elementary inequality
 $$
 \frac{x^\a}{t+x}\le t^{\a-1} \quad\mbox{for}\quad x\in[0,1],~t\in(0,1]
 $$
immediately implies that
 $$
 \left\|{X^\alpha}(tI +X)^{-1}\right\| \le t^{\alpha-1} \quad \text{and}\quad \left\|(tI + Y)^{-1}{Y^\beta}\right\|
 = \left\|\big({Y^\beta}(tI + Y)^{-1}\big)^*\right\|    \le t^{\beta -1},
 $$
 and so the right-hand side of \rf{eq:first_main_estimae} is less than or equal to
 $$
\big\|Y^{-\beta} (Y-X) X^{-\a}\big\|_{\bS_p} c_\s \int^1_0
t^{\s}t^{\alpha +\beta - 2}\,\rd t \le \frac{c_\s}{(\alpha + \beta + \s - 1)}
\big\|Y^{-\b}(Y-X)X^{-\a}\big\|_{\bS_p}\; .
$$

Next,
\begin{multline}
\label{eq:second_main_estimae}
c_\s \left\|\int^\infty_1t^{\s}(tI+Y)^{-1}(X-Y)(tI+X)^{-1}\rd t\right\|_{\bS_p} \\[.2cm]
\le
c_\s \int^\infty_1\big\|t^{\s}(tI+Y)^{-1}\big\|\cdot\|X-Y\|_{\bS_p}
\left\|(tI+X)^{-1}\right\|\rd t
\end{multline}
Obviously, for $t\ge1$,
$$
\big\|(tI+Y)^{-1}\big\|\le t^{-1}\quad\mbox{and}\quad\big\|(tI+X)^{-1}\big\|\le t^{-1}.
$$
Thus, the right-hand side of \rf{eq:second_main_estimae} is less than or equal to
$$
c_\s\|X-Y\|_{\bS_p}\int^\infty_1t^{\s - 2}\rd t \le \frac{c_\s}{1- \s}\|X-Y\|_{\bS_p}.
$$
Combining  inequality   \eqref{eq:first_est-te_for_Y1/2-X1/2}   with estimates \eqref{eq:first_main_estimae} and \eqref{eq:second_main_estimae},  we arrive at the conclusion   $Y^{\s} - X^{\s} \in \bS_p$. $\bl$

\begin{cor}
\label{cor:3.8}
Let $p>0$.
Suppose that
$\0 \le X\le I$ and $\0 \le Y\le I$. If $Y - X \in \bS_p$ and
$X$ is invertible, then $Y^\s-X^\s \in\bS_p$.
\end{cor}

\Pf If $X$ is invertible,
then $X^{-\alpha }$ is a bounded operator for every $\alpha \in\R$. Hence,
$(Y-X)X^{-\a}\in\bS_p$. Put $\b=0$. The result follows from Theorem \ref{Fyokla}. $\bl$

\medskip

For a contraction $T$, we consider the defect operators
$$
D_T \df(I - T^*T)^{1/2}\quad\mbox{and}\quad D_{T^*}\df(I - TT^*)^{1/2}.
$$

In what follows we agree that if $R$ is a linear operator defined on a dense subset of a Hilbert space and it extends to a bounded linear operator, we use the same notation $R$ for its unique extension to the whole space.

\begin{lem}
\label{raz def}
Let $p>0$. Suppose that
$T_0$ and $T_1$ are contractions such that
%ne nado$T- T_0 \in\bS_1$.
\bay
\label{yadro}
\Ker D_{T_0} = \{\0\}%\quad\mbox{and}\quad\Ker D_{T^*_0} = \{\0\}.
\ey
and for some  $\alpha\ge 0$, $\beta \ge 0$ satisfying  $\alpha  + \beta  \in(\tfrac{1}{2},1]$, the operators
\begin{equation}
\label{-2a}
D_{T_1^*}^{-2\b}(T_1-T_0)D_{T_0}^{-2\a}
\quad \mbox{and} \quad
D_{T_1}^{-2\b}(T_1^*-T^*_0)D_{T_0}^{-2\a}\quad  %% \mbox{for}\quad    \alpha  + \beta  \in(\tfrac{1}{2},1].
\end{equation}
are densely defined and extend to operators of class $\bS_p$.
Then
\bay
\label{defekty}
D_{T_1}-D_{T_0}\in \bS_p\quad\mbox{and}\quad D_{T_1^*} - D_{T^*_0} \in \bS_p.
\ey
\end{lem}

{\bf Remark 1.} Actually, $\Ker D_{T_0} = \{\0\}$ if and only if $\Ker D_{T^*_0} = \{\0\}$, and so \rf{yadro}
can be replaced with the condition
\bay
\label{vtoroeyadro}
\Ker D_{T^*_0} = \{\0\}.
\ey
Indeed, \rf{yadro} is equivalent to the fact that $\Ker(I-T_0^*T_0)= \{\0\}$ while \rf{vtoroeyadro} is equivalent to the condition
$\Ker(I-T_0T^*_0)= \{\0\}$. It remains to use the well-kown fact that $1$ is an eigenvalue of $T_0^*T_0$ if and only if it is an eigenvalue of $T_0T^*_0$.

\medskip

{\bf Remark 2.} Clearly, each of the inclusions in \rf{-2a} implies that $T_1- T_0 \in\bS_p$.

\medskip

\Pf
Put $Y\df I - T_1^*T_1 = D^2_{T_1}$ and $X\df I - T^*_0T_0 = D^2_{T_0}$. It is easily seen that
\begin{equation}
\label{eq:3.19}
Y-X = T^*_0T_0 - T_1^*T_1 = -(T_1^* -T^*_0)T_0 - T_1^*(T_1 - T_0) \in\bS_p.
\end{equation}
Clearly, we have the obvious commutation relation
$$
T_0D^2_{T_0} = D^2_{T^*_0} T_0
$$
which implies that
\bay \label{dlya polinomov}
T_0q(D^2_{T_0}) = q(D^2_{T^*_0})T_0
\ey
for an arbitrary polynomial $q$, and so \rf{dlya polinomov} also holds for an arbitrary continuous function $q$. In particular,
$$
T_0D_{T_0}^\alpha = D_{T^*_0}^\alpha  T_0, \quad\text{for every}\quad  \alpha \ge 0.
$$
 Hence,
$$
T_0f = D_{T^*_0}^{\alpha} T_0 D_{T_0}^{-\alpha}f, \qquad f \in \Range X^{\alpha},
$$
and so
\begin{equation}
\label{eq:3.23}
D_{T^*_0}^{-\alpha} T_0f = T_0 D_{T_0}^{-\alpha}f, \qquad f \in \Range X^\alpha.
\end{equation}
It follows from \eqref{eq:3.19} that
\begin{align*}
Y^{-\beta}(Y-X)X^{-\alpha}f
&= -Y^{-\beta}T_1^*(T_1 -T_0)D_{T_0}^{-2\alpha}f \\[.2cm]
&- Y^{-\beta}(T_1^* - T^*_0)T_0D_{T_0}^{-2\alpha}f,\quad f \in\Range X^\a.
\end{align*}
Together with \eqref{eq:3.23} this yields

\begin{align*}
Y^{-\beta}(Y-X)X^{-\alpha}f&
= -T_1^*D_{T_1^*}^{-2\beta}(T_1 -T_0)D_{T_0}^{-2\alpha}f\\[.2cm]
&-D_{T_1}^{-2\beta}(T^*_0 - T_1^*)D_{T_0^*}^{-2\alpha}T_0f,
\quad f \in\Range X^\a.
\end{align*}
Note that because of \rf{-2a} both operators on the right hand side of this identity are bounded
and $\Range X^\a$ is dense in $\h$. Hence,
$$
Y^{-\beta}(Y-X)X^{-\alpha} = -T_1^* D_{T_1^*}^{-2\beta}(T_1 -T_0)D_{T_0}^{-2\alpha}
- D_{T_1}^{-2\beta}(T^*_0 - T_1^*)D_{T_0^*}^{-2\alpha} T_0.
$$
Hence, with an account of condition   \rf{-2a} one gets  $Y^{-\beta}(Y-X)X^{-\alpha} \in \bS_p$.

Applying Theorem \ref{Fyokla} with $\s= 1/2$ ensures the inclusion
$Y^{1/2}-X^{1/2} \in \bS_p$ which means that $D_{T_1} - D_{T_0} \in \bS_p$.
Interchanging the roles of $T_1$ and $T_1^*$, we find that
$D_{T_1^*} - D_{T^*_0} \in \bS_p$. $\bl$

\medskip

The following lemma was established in \cite{MNP2}. We submit a proof here for completeness.

\begin{lem}
\label{vspomogalka}
Let $(T_0,T_1)$ be a pair of contractions with trace class difference. Suppose that
conditions {\em\rf{defekty}} hold for $p=1$. Then the pair $(T_0,T_1)$ has a real integrable spectral shift function.
\end{lem}

%Put $Y_\flat \df I - TT^* = D^2_{T_*}$.

%As in \eqref{eq:3.19}  we have
%\begin{equation}
%\label{eq:dif-ce_Y_*-X_*}
%Y_\flat - X_\flat = T_0T^*_0 - TT^* = -T(T^*-T^*_0) - (T-T_0)T^*_0.
%\end{equation}
%As above, starting with the commutation relation
%$$
%T^*_0X_\flat = XT^*_0,
%$$
%we find that
%$$
%X^{-\a}T^*_0f = T^*_0X_\flat^{-\alpha}f, \quad f \in \Range X^\alpha_\flat.
%$$
%Combining this equality with \eqref{eq:dif-ce_Y_*-X_*}, we obtain
%$$
%(Y_\flat - X_\flat)X^{-\alpha}_\flat f = -T(T^*-T^*_0)X^{-\alpha}_\flat  f- (T-T_0)X^{-\alpha}T^*_0f, \quad f\in \Range X^\alpha_\flat,
%$$
%which yields
%$$
%(Y_\flat - X_\flat)X^{-\a}_\flat = -T(T^*-T^*_0)X^{-\a}_\flat -
%(T-T_0)X^{-\a}T^*_0, \quad f\in \Range X^\a_\flat.
%$$
%Therefore, $[(Y_\flat - X_\flat)X^{-\alpha}_\flat] \in \bS_1$ and again, by Lemma \ref{lem:3.7},

%\begin{thm}
%\label{veshch}
%Suppose that $T_1$ and $T_0$ are contractions satisfying {\em{\rf{yadra}}} and
%{\em{\rf{-2a}}}. Then the pair $\{T_1,T_0\}$ has a real integrable spectral shift function.
%\end{thm}

To prove the lemma, we need the construction of Sch\"affer matrix dilation. Let $T$ be a contraction on a Hilbert space $\h$. We consider the Sch\"affer matrix dilation $U^{[T]}$ on the two-sided sequence space
$\ell^2_\Z(\h) = \bigoplus_{j \in \Z}\h_j$,  $\h_j = \h$,
of $\h$-valued sequences, see \cite{SNF}, Ch. 1, \S\:5.  We identify $\h$ with the subspace of sequences
$\{v_n\}_{n\in\Z}$ such that $v_j=\0$ for $j\ne0$.
This dilation does not have to be minimal. However, an advantage of this dilation is that it allows us to consider unitary dilations of contractions on $\h$ on the single space $\ell^2_\Z(\h)$.

The block matrix representation  of $U^{[T]}$ has the form:
\bay
\label{maShe}
U^{[T]}=
\left(\begin{matrix}
\ddots&\ddots&\vdots&\vdots&\vdots&\vdots&\vdots&\vdots&\iddots\\
\cdots&\0&I&\0&\0&\0&\0&\0&\cdots\\
\cdots&\0&\0&I&\0&\0&\0&\0&\cdots\\
\cdots&\0&\0&\0&D_T&-T^*&\0&\0&\cdots\\
\cdots&\0&\0&\0&T&D_{T^*}&\0&\0&\cdots\\
\cdots&\0&\0&\0&\0&\0&I&\0&\cdots\\
\cdots&\0&\0&\0&\0&\0&\0&I&\cdots\\
\iddots
&\vdots&\vdots&\vdots&\vdots&\vdots&\vdots&\ddots&\ddots
\end{matrix}\right).
\ey
Here the entry $T$ is at the $(0,0)$ position. In other words, the entries
$U^{[T]}_{j,k}$ of $U^{[T]}$ are given by
$$
U^{[T]}_{0,0}=T,\quad U^{[T]}_{0,1}=D_{T^*},\quad U^{[T]}_{-1,0}=D_T,
\quad U^{[T]}_{-1,1}=-T^*,\quad
U^{[T]}_{j,j+1}=I\quad\mbox{for}\quad j\ne0,~-1,
$$
while all the remaining entries are equal to $\0$.

\medskip

{\bf Proof of Lemma \ref{vspomogalka}.}
Consider the Sch\"affer matrix dilations $U^{[T_0]}$ and $U^{[T_1]}$
of $T_0$ and $T_1$.

It follows easily from \eqref{maShe}
and \rf{defekty}
that  $U^{[T_1]}-U^{[T_0]} \in\bS_1$. By Krein's theorem \cite{Kr2}, the pair
$\Big(U^{[T_0]}, U^{[T_1]}\Big)$ has a real spectral shift function $\bs\xi$ in $L^1(\T)$, i.e., the following trace formula holds
$$
\trace\Big(f\big(U^{[T_1]}\big)-f\big(U^{[T_0]}\big)\Big)=\int_\T f'(\z)\bs\xi(\z)\,\rd\z
$$
for an arbitrary trigonometric polynomial $f$.

It is easy to see that if $f$ is an analytic polynomial, then the $(0,0)$ entries  $f\big(U^{[T_1]}\big)_{0,0}$ and $f\big(U^{[T_0]}\big)_{0,0}$
are given by
$$
f\big(U^{[T_1]}\big)_{0,0}=f(T_1)\quad\mbox{and}\quad f\big(U^{[T_0]}\big)_{0,0}=f(T_0).
$$

It follows that for an arbitrary analytic polynomial $f$,
$$
\trace\big(f(T_1)-f(T_0)\big)
=\trace\Big(f\big(U^{[T_1]}\big)-f\big(U^{[T_0]}\big)\Big)=\int_\T f'(\z)\bs\xi(\z)\,\rd\z.
$$
Thus, $\bs\xi$ is a real integrable spectral shift function of the pair $\{T_0,T_1\}$.
$\bl$

%\medskip
%
%Note that the observation that under conditions \rf{defekty}
%each spectral shift function for the pair $\{U^{[T_0]},U^{[T_1]}\}$ must be
%a spectral shift function for the pair $\{T_0,T_1\}$ was made in \cite{MNP2}.

\begin{thm}
\label{veshch}
Suppose that $T_1$ and $T_0$ are contractions satisfying {\em{\rf{yadro}}} and
{\em{\rf{-2a}}} for $p=1$. Then the pair $(T_1,T_1)$ has a real integrable spectral shift function.
\end{thm}

\Pf The result follows immediately from Lemma \ref{raz def} and Lemma \ref{vspomogalka}. $\bl$

%\begin{cor}
%Let $T$ and $T_0$ be contractions such that $T- T_0 \in \bS_1$. If $1\in \rho(T_0)$, then
%the pair $\{T,T_0\}$ has a real-valued integrable spectral shift function.
%\end{cor}
%
%\Pf
%The condition $1\in \rho(T_0)$ is equivalent to the bounded invertibility of each of the defect operators $D_{T_0}$ and
%$D_{T^*_0}$.
%Therefore  the pair $\{T,T_0\}$ meets condition  \eqref{eq:3.18}. Theorem \ref{veshch} completes the proof.
%%%The result follows from Corollary \ref{cor:3.8}.
%$\bl$

\medskip

The following corollary shows that  our Theorem  \ref{veshch} yields the existence of a real-valued spectral shift function under much weaker assumptions than the main result of \cite{CS} that has been mentioned in the introduction.

\begin{cor}
Let $T$ and $T_0$ be contractions such that $T_1- T_0 \in \bS_1$. If $\|T_0\| < 1$, then
the pair $(T_0,T_1)$ has a real-valued spectral shift function.
\end{cor}

\Pf If $\|T_0\| < 1$, then $D_{T_0}$ and $D_{T^*_0}$ are invertible. The result follows from
Corollary \ref{cor:3.8}. $\bl$

\medskip

Next, we present a result obtained in another way by the authors in [MNP2], Lemma 9.5.

\begin{cor} Suppose that $T_0$ is a Fredholm contraction of zero index. Then there exists an
invertible contraction $T_1$ such that $T_1-T_0\in\bS_1$ and the pair $(T_0,T_1)$ has a real-valued
spectral shift function.
\end{cor}

\Pf We define a perturbation V like in [MNP2] as a partial isometry with 
initial space $\Ker T_0$ and final space $\Ker T^*$. Then the operator $T_1=T_0+V$ is a
desired one. Indeed, $T_1$ is an invertible contraction, and so the pair
$(T_0,T_1)$ satisfies the hypotheses of Theorem \ref{veshch} with $\a = 0$ and $\b = 1$.

\begin{cor}
Let $T$ and $X$ be contractions such that $X\ge\0$ and $I-X\in\bS_1$. Suppose that the following conditions hold:
$$
\Ker D_T = \{\0\}
$$
and
$$
(I-X)TD_T^{-2\a}\in\bS_1
\quad \mbox{and} \quad
T^*(I-X)D_{T^*}^{-2\a}\in \bS_1
$$
for a number $\a$ in $(\tfrac{1}{2},1]$.
Then the pair $(T,XT)$ has both a real-valued integrable spectral shift function and a purely imaginary
integrable 
spectral shift function.
\end{cor}

\Pf Indeed, the first part of the conclusion is an immediate consequence of Theorem \ref{veshch}
while the second part is the essence of Lemma 9.2 of the paper \cite{MNP2}.
$\bl$

\medskip

Let us also observe that if $\bs\xi$ is a real-valued spectral shift function and $\ri\bs\eta$ is a purely imaginary
spectral shift function, then the function  $\bs\xi-\ri\bs\eta$ belongs to the Hardy class $H^1$, and so the harmonically conjugated functions $\widetilde{\bs\xi}$ and $\widetilde{\bs\eta}$ must be integrable.

If, in addition to this, $\bs\xi\ge\0$, then the function $\bs\xi$ satisfies the Zygmund condition
$$
\int_\T\bs\xi(\z)\log(1+\bs\xi(\z))\,{\rm d}\m(\z)<\be.
$$
Recall in this connection that in the papers \cite{MN} and \cite{MNP1} it was observed that if  a pair of contractions with trace class difference possesses a complex-valued spectral shift function that satisfies
the Zygmund condition, then it also has a real-valued integrable spectral shift function.

\

\section{\bf The case of dissipative operators}
\label{dissipatsiya}

\

As we have mentioned in the introduction if $(L_0,L_1)$ is a pair of maximal dissipative operators with trace class resolvent difference, i.e.,
\begin{equation}
\label{eq:res,dif-ce,of,trace,cl}
(L_1+ \ri I)^{-1} - (L_0 +\ri I)^{-1} \in \bS_1,
\end{equation}
trace formula \rf{fcddo} holds with a spectral shift function $\bs\xi$ satisfying \rf{s_vesom}.

In this section we find a sufficient condition for the existence of a {\it real-valued} spectral shift function $\bs\xi$ satisfying \rf{s_vesom}, see Theorem \ref{imL0} below.

Recall also that if we replace condition \rf{eq:res,dif-ce,of,trace,cl} with the condition
\bay
\label{yadrazn}
L_1-L_0\in\bS_1,
\ey
then trace formula \rf{fcddo} holds with an {\it integrable} spectral shift function $\bs\xi$. Recall that both results were obtained in \cite{MNP2}.

However, it turns out that under condition \rf{yadrazn}, an {\it integrable} real-valued spectral shift function
does not have to exist. Moreover, such a function cannot exist if
$\trace(L_1-L_0)\not\in\R$, see Theorem \ref{net_takikh} below.

%for the existence of a real-valued spectral shift function. However, our condition implies the existence of a real-valued spectral shift function $\bs\xi$ that satisfies inequality \rf{s_vesom}. Moreover, it turns out that under our assumptions a real-valued integrable spectral shift function does not have to exist. 

With each maximal dissipative operator $L$ we associate the quadratic form
${\frak T}^{\im}_{L}[f]=\im (Lf,f)$, $f\in {\rm dom}(L)$.
This form is non-negative. Though it is not necessarily closable.
%In what follows we always assume that the form ${\frak T}^{\im}_{L}[f]$ is closable
%and denote by $\im L$ the non-negative self-adjoint  operator associated with its closure $\clos {\frak T}_{L,\im}$
%in accordance with the first representation theorem (see \cite{Ka}..........)

In what follows we always assume that the symmetric form ${\frak T}^{\im}_{L}[f]$ is closable,
and so by the first representation theorem (see \cite[Ch.VI, Th. 2.1]{Ka}),
its closure $\clos {\frak T}_{L,\im}$  corresponds to a  nonnegative self-adjoint  operator,
which we denote by   $\im L$.

If ${\rm dom}(L) = {\rm dom}(L^*)$, then the operator $\im L$ is just the Friedrichs extension of the
symmetric operator  $\frac{1}{2\ri}(L-L^*) \ge 0$ (see \cite[Ch.VI, Theorem 2.9]{Ka}).

If the numerical range of ${\frak T}^{\im}_{L}[f]$  is contained in the sector 
$$
S_\vt = \{z\in \Bbb C_+:
\im z \ge \cot\vt |\re z| \},
$$
then the form ${\frak T}^{\im}_{L}[f]$ is necessarily closable
and the operator $\im L$  definitely exists, it is  nonnegative and self-adjoint (see \cite[Ch.VI, Th. 1.27]{Ka}).

Recall that if $T$ and $R$ are linear operator in Hilbert space with domains $\dom(T)$ and $\dom(R)$,
the operator $R$ is said to be {\it dominated by} $T$ if
\bay
\label{1v2}
\dom(T)\subset\dom(R)
\ey
and
\bay
\label{podchi}
\|Rx\|^2\le a\|Tx\|^2+b\|x\|^2,\quad x\in\dom(T),
\ey
for some positive numbers $a$ and $b$. If $R$ is dominated by $T$, the number
$$
{\rm a}_T(R)\df\inf\{a>0:~\mbox{ there exists positive}~b~\mbox{ such that }~\rf{1v2}~\mbox{ holds }\}
$$
is called the {\it $T$-bound of $R$}.

It is well known (see \cite{BS2}, \S\:3.4) that if $T$ and $R$ are closed operators satisfying \rf{1v2}, then inequality \rf{podchi} holds for some positive numbers $a$ and $b$.

It is also easy to verify that if $L$ is a maximal dissipative operator, then an operator $R$ is dominated by $L$ if and only if the operator $R(L+\ri I)^{-1}$ is bounded.

\begin{thm}
\label{imL0}
Let $L_0$ and $L_1$ be 
maximal dissipative operators satisfying {\em{\rf{eq:res,dif-ce,of,trace,cl}}} such that
$\dom(L_0) = {\rm dom}(L_1)$ and $\Ker\im L_j=\{\0\}$, $j=0,1$. 
%and suppose that 
%\begin{equation}
%%\label{eq:res,dif-ce,of,trace,cl}
%(L_1+ \ri I)^{-1} - (L_0 +\ri I)^{-1} \in \bS_1.
%\end{equation}
 Assume also that both forms $\frak T^{\im}_{L_j}$  are closable
 and the operator $(\im L_j)^{1/2}$ is dominated by $L_j + \ri I$ and by $L_j^* - \ri I$, i.e.,
the operators
   \begin{equation}
   \label{eq:im,L_j,sub-d,to_L_j}
 (\im L_j)^{1/2}(L_j  + \ri I)^{-1} \quad \text{and} \quad (\im L_j)^{1/2}(L_j^*  - \ri I)^{-1} \quad  \text{are bounded},  \quad j\in \{0,1\}.
    \end{equation}
Assume in addition that

\bay
\label{Im-1}
(\im L_1)^{-1/2} \big(L_1 - L_0 \big)(\im L_0)^{-1/2} \in \bS_1
\ey
in the sense that the operator in {\em\rf{Im-1}} is densely defined and extends by continuity to a trace class operator.

Then the pair $(L_0,L_1)$ has  a real-valued  
spectral shift function $\bs\xi$ satisfying the condition
\bay
\label{skvadratom}
\int_\R\frac{|\bs{\xi}(t)|}{1+t^2}\,\rd t<\be.
\ey
\end{thm}

%\begin{cor}
%If both $\im L_0$  and $\im L_1$ are boundedly invertible,
%then condition  \eqref{Im-1} becomes obsolete.
%\end{cor}

{\bf Proof of Theorem  \ref{imL0}.}
Consider the Cayley transforms of $L_0$ and $L_1$ defined by
\begin{equation}
\label{eq:Cayley_transf-s}
T_0 = (L_0 -\ri I)(L_0 + \ri I)^{-1} \quad \text{and} \quad T_1 = (L_1 -\ri I)(L_1 + \ri I)^{-1}.
\end{equation}
Let us show that the contractions $T_0$ and $T_1$ satisfy the hypotheses
of Theorem  \ref{veshch} with $\a = \b =1/2$ and $p=1$.

It follows easily from \eqref{eq:Cayley_transf-s}  that
\begin{eqnarray}
\label{eq:ident_for_T-T_0}
T_1 - T_0 = -2\ri\big((L_1+ \ri I)^{-1} - (L_0 +\ri I)^{-1}\big).
\end{eqnarray}

Combining this identity with \eqref{eq:res,dif-ce,of,trace,cl},
we see that $T_1- T_0 \in \bS_1$.  Put
   \begin{eqnarray}\label{eq:oper,G_j,wtG_j}
G_j = (\im L_j)^{1/2}(L_j + \ri I)^{-1}\quad \text{and} \quad  \widetilde G_j = (L_j +\ri I)^{-1}(\im L_j)^{1/2}, \quad j\in \{0,1\}.
     \end{eqnarray}
By \eqref{eq:im,L_j,sub-d,to_L_j},  the operators $G_j$ are bounded. Since
$$
(L_j +\ri I)^{-1}(\im L_j)^{1/2} \subset  \big((\im L_j)^{1/2}(L_j^*  - \ri I)^{-1}\big)^*,
$$
the operators $(L_j +\ri I)^{-1}(\im L_j)^{1/2}$ are also
bounded.

With this notation, one can easily deduce from \eqref{eq:Cayley_transf-s} that
\begin{align}
\label{eq:ident_for_D_T_0}
D^2_{T_j}&= I - T^*_j T_j  =2\ri\big((L_j +\ri I)^{-1} - (L^*_j -\ri I)^{-1} \nonumber\\
& - 2\ri(L^*_j -\ri I)^{-1}(L_j +\ri I)^{-1}\big)  \nonumber  \\
&= 4\big((L^*_j -\ri I)^{-1} (\im L_j)^{1/2}\big) (\im L_j)^{1/2} (L_j +\ri I)^{-1} \nonumber  \\
&= 4\big((\im L_j)^{1/2}(L_j +\ri I)^{-1}\big)^*
(\im L_j)^{1/2}(L_j +\ri I)^{-1}= 4G_j^* G_j.
 \end{align}
Since $\Ker\im L_j=\{\0\}$, it obviously follows that $\Ker G_j=\Ker G_j^*G_j = \{\0\}$, and so
$T_0$ satisfies  \rf{yadro}.

Next, \eqref{eq:ident_for_D_T_0} implies that $2|G_j| = 2(G_j^* G_j)^{1/2} =  D_{T_j}$, $j\in\{0,1\}$, and hence the polar decomposition of $G_j$  is  $G_j = V_j|G_j| = 2^{-1}V_j D_{T_j}$  where $V_j$, $j\in\{0,1\}$, is unitary.

Similarly, 
\eqref{eq:oper,G_j,wtG_j} that
\begin{align}
\label{eq:ident_for_D_T_j*}
D^2_{T_j^*}& = I - T_j T^*_j   = 2\ri\big((L_j +\ri I)^{-1} - (L^*_j -\ri I)^{-1} \nonumber\\
& - 2(L_j +\ri I)^{-1}\ri(L^*_j -\ri I)^{-1}\big)  \nonumber  \\
&= 4(L_j +\ri I)^{-1} \im L_j (L^*_j -\ri I)^{-1}  \nonumber  \\
&= %%4(L^*_j -\ri I)^{-1}(\im L_j)^{1/2}
4(L_j +\ri I)^{-1}(\im L_j)^{1/2} 
(\im L_j)^{1/2}(L_j^* - \ri I)^{-1}= 4\widetilde G_j \widetilde G_j^* ,
 \end{align}
It follows that $|\widetilde G_j^*| = (\widetilde G_j \widetilde G_j^*)^{1/2} =  2^{-1} D_{T_j^*}$, $j=0,1$.

We also need the  antipolar  representation of $\widetilde G_j$:
$\widetilde G_j = |\widetilde G_j^*| \widetilde V_j = 2^{-1}D_{T_j^*}\widetilde V_j$ (see \cite[Ch. 8]{BS2}, \cite{Ka}). Note that the $\widetilde V_j$, $j\in\{0,1\}$, are unitary. Thus, combining the polar and the antipolar decompositions
for $G_0$  and $\widetilde G_1$ and keeping  \eqref{eq:oper,G_j,wtG_j}  in mind, we easily  obtain
\bay
   \label{nulevoi}
2D_{T_0}^{-1} = G_0^{-1}V_0 = (L_0 + \ri I)(\im L_0)^{-1/2} V_0
\ey
and
\bay
\label{pervyi}
2D_{T_1^*}^{-1} =\widetilde  V_1 (\widetilde G_1)^{-1} = \widetilde  V_1 (\im L_1)^{-1/2}(L_1 +\ri I).
\ey

Multiplying \eqref{eq:ident_for_T-T_0} by $D_{T_1^*}^{-1}$ on the left and by $D_{T_0}^{-1}$
on the right, using  representations \rf{nulevoi} and \rf{pervyi}
for $D_{T_0}^{-1}$ and $D_{T_1^*}^{-1}$ and keeping \eqref{Im-1} in mind, we obtain
\begin{eqnarray*}
D^{-1}_{T_1^*}(T_1 - T_0)D^{-1}_{T_0} = 2^{-1}\ri D^{-1}_{T_1^*} \big((L_1+ \ri I)^{-1} - (L_0 +\ri I)^{-1}\big)D^{-1}_{T_0} \nonumber \\
=  2^{-1}\ri D^{-1}_{T_1^*} (L_1 +\ri I)^{-1} (L_1 - L_0)(L_0 + iI)^{-1}D^{-1}_{T_0}  \nonumber \\
= 2^{-1}\ri \widetilde V_1(\im L_1)^{-1/2} (L_1- L_0) (\im L_0)^{-1/2}V_0 \in\bS_1  
   \end{eqnarray*}
by \rf{Im-1}. Similarly, 
$$
D^{-1}_{T_1}(T_1^* - T^*_0)D^{-1}_{T^*_0}\in\bS_1.
$$
Thus, $T_0$ and $T_1$ satisfy \rf{-2a} with with $\a = \b =1/2$.
By Theorem \ref{veshch}, the pair of contractions
$\{T_0,T_1\}$ has an integrable real-valued spectral shift function $\bs{\xi}_{\rm c}$.
We can define now the function $\bs\xi$ by
$$
\bs{\xi}(t)\df\bs{\xi}_{\rm c}\left(\frac{t-\ri}{t+\ri}\right),\quad t\in\R.
$$
Clearly, \rf{skvadratom} is satisfied and trace formula \rf{fcddo} holds
for an arbitrary rational function $f$ with poles in $\C_-=\{\z\in\C:~\im\z<0\}$. That is, $\bs\xi$ is a real-valued spectral shift function for the pair $\{L_0,L_1\}$.
$\bl$

\medskip

Finally, we deduce
from Theorem 5.6  \cite{MalNei2015} the following result.
%\cite{MN2}.

\begin{thm}
\label{net_takikh}
Suppose that $L_0$ and $L_1$ are maximal dissipative operators satisfying {\em\rf{yadrazn}} and such that
$$
\trace(L_1-L_0)\not\in\R.
$$
Then the pair $(L_0,L_1)$ does not possess a real-valued integrable function.

Moreover, if $L_1-L_0$ is dissipative, then the condition $\trace(L_1-L_0)\in\R$ is equivalent to the existence of an integrable real-valued spectral shift function.
\end{thm}

Before we proceed to the proof, let us write the following special case of trace formula \rf{fcddo}:
\bay
\label{fcddo,new}
\trace\big((L_1- zI)^{-1} - (L_0 - zI)^{-1}\big) = \int_\R \frac{\bs{\xi}(t)}{(t-z)^2}\,\rd t
\ey

\Pf Put $V\df L_1-L_0$.

It is easy to see that under \rf{yadrazn} identity \eqref{fcddo,new} implies that
\bay
\label{eq:trV=int,xi}
\trace\big((L_1- zI)^{-1} V (L_0 - zI)^{-1}\big) = \int_\R \frac{\bs{\xi}(t)}{(t-z)^2}\,\rd t.
\ey
Setting here  $z=-iy$, multiplying both sides by $y^2$ and passing to the limit as $y\to\infty$, we find that
$$
\trace V  = \int_\R \bs{\xi}(t)\,\rd t.
$$
Both conclusions in the statement of the theorem follow immediately from this formula. $\bl$
  
\

\section{\bf An application to dissipative Schr\"odinger operators}
\label{applic-n}

\

In this section we apply  Theorem  \ref{imL0}  to the pair $\{L_0, L_1\}$ of maximal dissipative  Schr\"odinger operators, where
     \begin{equation}
     \label{eq:def,of,Schr,oper-s}
L_0 = (1 + \ri)D^2 + \ri I \qquad \text{ and} \qquad L_1 = (1 + \ri)D^2 + M_{\ri+q}
    \end{equation}
in $L^2(\Bbb R)$ such that the potential $\ri+q$ satisfies  $\im q\ge\0$. Here $D= -\ri \frac{\rd}{\rd x}$ and by $M_\f$ we understand multiplication by $\f$, i.e., 
$$
M_\f f=\f f.
$$
In other words,
$$
L_1 f=-(1+\ri)f''+(\ri+q)f.
$$

For a compact operator $T$ on Hilbert space, we consider the sequences $\{s_j(T)\}_{j\ge0}$ and 
$\{\l_j(T)\}_{j\ge0}$
the sequences of singular values and the sequence of eigenvalues of $T$ counted with algebraic multiplicities and we arrange them so that
$$
s_{j+1}\le s_j\quad\mbox{and}\quad|\l_{j+1}|\le|\l_j|,\quad j\ge0.
$$
If $T$ has finitely many eigenvalues, we extend the finite sequence of its eigenvalues with infinitely many zero terms.

We are going to use the following fact (see \cite{BS2}, \S\:3.10, Th. 2):

Let $A$ and $B$ be compact operators on Hilbert space. Then
$$
\l_j(AB)=\l_j(BA),\quad j\ge0.
$$

\begin{lem}
Suppose that $T_1$ and $T_2$ are compact self-adjoint operators on Hilbert space.
Then 
$$
s_j(T_1T_2) = s_j(T_2 T_1),\quad j\ge0.
$$ \end{lem}
\Pf
Clearly,
$$
\lambda_j\left((T_1T_2)^*(T_1T_2)\right) = \lambda_j\left(T_2 T_1^2T_2\right) = \lambda_j\left(T_1^2T_2^2\right)
$$
Similarly, 
$$
\lambda_j\left((T_1T_2)(T_1T_2)^*\right) = \lambda_j\left(T_1 T_2^2T_1\right) = \lambda_j\left(T_1^2T_2^2\right)
$$
Thus,
$s_j(T_1T_2)^2 = \lambda_j\left((T_1T_2)^*(T_1T_2)\right)=s_j(T_2T_1)^2$. $\bl$

\medskip

We need the following result. 

\begin{thm}
\label{Kuzya}
Let $q$ be a nonnegative function in $L^1(\R)\cap L^2(\R)$ and let $R$ be a continuous positive-definite function on $\R\times\R$.
Suppose that
$$
\int_\R q(x)R(x,x)\,\rd x<\be.
$$
Then the integral operator $K$ on $L^2(\R)$ defined by
$$
(Kf)(s)=\int_\R q(s)^{1/2}R(s,t)q(t)^{1/2}f(t)\,\rd t
$$
is a nonnegative operator of trace class and
$$
\|K\|_{\bS_1}=\trace K=\int_\R q(x)R(x,x)\,\rd x<\be.
$$
\end{thm}

Note that in the case when the function $q$ is continuous the result is well-known, see e.g. \cite[Section 3.10]{GK}. We give two different proofs of Theorem \ref{Kuzya}.

\medskip

{\bf The first proof of Theorem \ref{Kuzya}.} First of all, it suffices to assume that both $R$ and $q$ are compactly supported. Indeed, let $w$ be a continuous function $w$ on $\R$ such that
$$
0\le w(t)\le1\quad\mbox{for}\quad t\in\R;\qquad\supp w=[-2,2]\qquad\mbox{and}
\qquad w(t)=1\quad\mbox{for}\quad t\in[-1,1].
$$
We can define the operators $K_n$, $n\ge1$ by
$$
(K_nf)(s)=\int_\R q(s)^{1/2}R_n(s,t)q(t)^{1/2}f(t)\,\rd t,
$$
where
$$
R_n(s,t)\df w\Big(\frac sn\Big)R(s,t)w\Big(\frac tn\Big).
$$
Clearly, $R_n$ is a positive definite function with compact support.

By the assumption that the conclusion of the theorem holds for compactly supported $R$ and $q$, we can conclude that $K_n\in\bS_1$ and
$$
\|K_n\|_{\bS_1}=\trace K_n=\int_\R q(x)R_n(x,x)\,\rd x<\be.
$$
It is easy to se that the sequence $\{K_n\}_{n\ge1}$ converges in the weak operator topology and
$\mbox{w-}\lim_{n\to\be}K_n=K$. Also, it is clear that $(K_nf,f)\le(Kf,f)$ for $f\in L^2(\R)$ and $n\ge1$.
It follows that
$$
\lim_{n\to\be}\trace K_n=\lim_{n\to\be}\sum_{j\ge1}(K_nf_j,f_j)=\sum_{j\ge1}(Kf_j,f_j)=\trace K,
$$
where $\{f_j\}_{j\ge1}$ is an orthonormal basis in $L^2(R)$.

Thus, we assume that both $R$ and $q$ are compactly supported. It follows that the operator $K$ belongs to the Hilbert--Schmidt class $\bS_2$. Let us establish that $K\in\bS_1$.

Consider an infinitely smooth function $v$ on $\R$ such that
$$
v(t)\ge0\quad\mbox{for}\quad t\in\R;\qquad\supp v=[-1,1],\qquad\mbox{and}\qquad\int_\R v(t)\rd t=1.
$$
Put $V_n(t)=nv(nt)$, $t\in\R$. Consider the function $\Phi_n$ on $\R^2$ defined by
$$
\Phi_n(s,t)=\int_\R V_n(u)q^{1/2}(s-u)R(s-u,t-u)q^{1/2}(t-u)\,\rd u.
$$
It is easy to see that $\Phi_n$ is a compactly supported infinitely smooth function and $\Phi_n$ is positive definite for every $n$.

Consider now the integral operator $T_n$ on $L^2(\R)$ defined by
$$
(T_nf)(s)=\int_\R\Phi_n(s,t)f(t)\,\rd t.
$$
Since $\Phi_n$ is a compactly supported infinitely smooth function, it follows that $T_n\in\bS_1$ for every $n\ge1$. It follows from \cite{B} that
$$
\trace T_n=\|T_n\|_{\bS_1}=\int_\R\Phi_n(s,s)\,\rd{s}.
$$
It is easy to verify that
$$
\lim_{n\to\be}\int_\R\Phi_n(s,s)\,\rd{s}=\int_\R q(s)R(s,s)\,\rd x.
$$
It is also clear that
$$
\lim_{n\to\be}\|K-T_n\|_{\bS_2}=0.
$$
This easily implies that $K\in\bS_1$.

Once we know that $K\in\bS_1$, we can easily conclude that
\bay
\label{yaderno!}
\lim_{n\to\be}\|K-T_n\|_{\bS_1}=0.
\ey

Indeed, this is a consequence of the following lemma.

\begin{lem}
\label{Puzya}
Let $R$ be a trace class operator on $L^2(\R)$ and
$$
(Rf)(s)=\int_\R k(s,t)f(t)\,\rd t,
$$
for a function $k$ in $L^2(\R^2)$. Let $R_n$, $n\ge1$, be the operators defined by
$$
(R_nf)(s)=\int_\R k_n(s,t)f(t)\,\rd t,
$$
where
$$
k_n(s,t)=\int_\R V_n(u)k(s-u,t-u)\,\rd u.
$$
Then $R_n\in\bS_1$ and
$$
\lim_{n\to\be}\|R_n-R\|_{\bS_1}=0.
$$
\end{lem}

Let us first complete the proof of Theorem \ref{Kuzya}. Since $\Phi_n$ is a smooth function with compact support, it is well known (see \cite{B}), that
$$
\trace T_n=\int_\R\Phi_n(s,s)\,\rd s.
$$
To complete the proof, we observe that it follows from \rf{yaderno!} that
$$
\trace K=\lim_{n\to\be}\trace T_n.\quad\bl
$$

\medskip

{\bf Proof of Lemma \ref{Puzya}.} Clearly, it suffices to prove the result for rank one operators $R$ in which case the result is trivial. $\bl$

\medskip

Let us proceed now to the second proof of Theorem \ref{Kuzya}.

First we complement Theorem 3.5.1 from \cite{GK}. Suppose that $\{A_n\}_1^\infty$ is a sequence of operators on Hilbert space of class $\bS_p$, $1\le p<\be$, that converges in the norm to a bounded linear operator $A$ and
 \begin{equation}
   \label{eq:sup,A_n<infty}
 \sup_n \|A_n\|_{\bS_p} < \infty.
   \end{equation}
It is well known that in this case $A$ must be in $\bS_p$ and its $\bS_p$ norm is less than or equal to the supremum in 
\rf{eq:sup,A_n<infty}. The following lemma shows that if we in addition to this assume that the sequence $\{|A_n|\}_1^\infty$ of the moduli of $A_n$ is monotone, we can make a stronger conclusion.

\begin{lem}\label{lem:conv,S_1-pos,oper}
Suppose that in addition to the above assumptions the sequence $\{|A_n|\}_1^\infty$ is monotone. Then
    \begin{equation}\label{eq:|A_n-A|,to,0}
\lim_{n\to \infty} \|A - A_n\|_{\bS_p} 
= 0.
   \end{equation}
Moreover, if the operators $A_n$ are nonnegative, then

   \begin{equation}
   \label{eq:A_n>0,|A|_p=sum,l_j}
\lim_{n\to \infty} \|A_n\|_{\bS_p} = \|A\|_{\bS_p}% = \left(\sum_j \lambda_{j}(A)^p\right)^{1/p} = (\tr A^p)^{1/p}.
   \end{equation}
\end{lem}
\Pf
%It is well known (see e.g. .....) that
%$$
%|s_{j}(A_n) - s_{j}(A)| \le \|A_n - A\| \quad \text{for each}\quad j, n\in \Bbb N.
%%% |\lambda_{j}(A_n) - \lambda_{j}(A)|
%$$
%Therefore for each fixed $j\in \Bbb N$ the uniform convergence $\lim_{n\to\infty}\|A_n - A\| =0$ implies
%$\lim_{n\to \infty} s_{j}(A_n) = s_{j}(A)$.
%Applying the Fatou lemma and using  \eqref{eq:sup,A_n<infty},  we obtain
%$$
%\sum_{j\in \Bbb N} s_{j}(A)^p = \sum_{j\in \Bbb N}\lim_{n\to \infty} s_{j}(A_n) \le 
%{\lim \inf}_{n\to \infty}\sum_{j\in \Bbb N}s_{j}(A_n)^p \le  
%\sup_n \sum_{j\in \Bbb N} s_{j}(A_n)^p < \infty.
%$$
%
Assume to be definite that the sequence $\{|A_n|\}$ is increasing. Then $s_{j}(A_n) \uparrow s_{j}(A)$ for each $j\in \Bbb N$ 
as $n\to\infty$.
Therefore, applying the monotone convergence theorem, we obtain \rf{eq:A_n>0,|A|_p=sum,l_j}.
 
It we replace now $A_n$ with $A - A_n$  and observe that the resulting sequence is also monotone,
we arrive at \eqref{eq:|A_n-A|,to,0}. $\bl$

%////////////////////////  MMM  //////////////  MMM  /////////   \\

%\begin{thm}\label{KuzyaNew}
%Let $q$ be a nonnegative function in $L^1(\R)\cap L^2(\R)$ and let $R$ be a continuous positive-definite function on $\R\times\R$.
%%%Suppose that
%%%$$
%%%\int_\R q(x)R(x,x)\,\rd x<\be.
%%$$
%Then the integral operator $K$ on $L^2(\R)$ defined by
%$$
%(Kf)(s)=\int_\R q(s)^{1/2}R(s,t)q(t)^{1/2}f(t)\,\rd t
%$$
%is a nonnegative operator of trace class and
%$$
%\|K\|_{\bS_1}=\trace K=\int_\R q(x)R(x,x)\,\rd x<\be.
%$$
%\end{thm}

\medskip

{\bf The second proof of Theorem \ref{Kuzya}.}
It suffices to assume that the function $q$ is supported on a compact interval $[a,b]$.
Choose a monotone sequence $\varphi_n(\cdot)\in C[a,b]$  approximating $q$ from below in $L^1[a,b]$:
  \begin{equation}\label{monot,seq,phi_n}
\varphi_n(\cdot)\in C[a,b], \quad \varphi_n(t)\le \varphi_{n+1}(t) \le q(t)\quad and \quad \lim_{n\to \infty}\|\varphi_n - q\|_{L^1}=0
   \end{equation}

Define the operators $K_n$, $n\ge1$, by setting
     \begin{equation}\label{eq:def,K_n}
(K_nf)(s) = \int_\R \varphi_n(s)^{1/2}R(s,t)\varphi_n(t)^{1/2}f(t)\,\rd t,
   \end{equation}
   
Clearly,
   \begin{equation}\label{eq:K_n,to,K,in,S_2-norm}
%%\lim_{n\to\be}
\|K_n - K\|_{\bS_2}^2 = \int_a^b\int_a^b |(\varphi_n(s) \varphi_n(t))^{1/2} - ( q(s)q(t))^{1/2}|^2\cdot|R(s,t)|^2\,\rd s\,\rd t \to 0
\quad \text{as}\quad n\to \infty.
   \end{equation}
Further, since the integral operator $K_n$ given by \eqref{eq:def,K_n} is non-negative with continuous kernel, and
$\int_a^b \varphi_n(t)R(t,t)\rd t < \infty$, it is  of trace
 class,  $K_n \in \bS_1$,  and
      \begin{equation}\label{eq:S_1-norm,of,K_n}
\|K_n\|_{\bS_1} = \tr K_n = \int_a^b \varphi_n(t)R(t,t)\rd t, \quad n\in \Bbb N
      \end{equation}
(see \cite[Chapter 3.10]{GK}).
In turn, it follows from  \eqref{eq:S_1-norm,of,K_n} and  \eqref{monot,seq,phi_n} that
      \begin{equation}\label{eq:S_1-norm,of,K_n,second}
\sup_n \|K_n\|_{\bS_1} = \sup_n\int_a^b \varphi_n(t)R(t,t)\rd t = \int_a^b q(t)R(t,t)\rd t < \infty.
    \end{equation}
Combining this estimate with relation \eqref{eq:K_n,to,K,in,S_2-norm} and applying \cite[Theorem 3.5.1]{GK} yields
the inclusion $K \in \bS_1$ and the following estimate
$$
\|K\|_{\bS_1} \le \sup_n \|K_n\|_{\bS_1} \left( = \int_a^b q(t)R(t,t)\rd t \right).
$$
Thus, condition  \eqref{eq:sup,A_n<infty} with $p=1$ %%of Lemma \ref{lem:conv,S_1-pos,oper}
is verified and due to  \eqref{eq:K_n,to,K,in,S_2-norm}  $\lim_{n\to\infty} \|K_n - K\|=0$.
Thus, we can apply Lemma \ref{lem:conv,S_1-pos,oper} to the sequence $\{K_n\}_1^\infty$ which together with
 \eqref{eq:S_1-norm,of,K_n}--\eqref{eq:S_1-norm,of,K_n,second} yields
 %%\eqref{eq:K_n,to,K,in,S_2-norm}
 %%Besides, one gets from \eqref{eq:S_1-norm,of,K_n} and \eqref{monot,seq,phi_n} that

$$
\|K\|_{\bS_1} =  \lim_{n\to \infty} \|K_n\|_{\bS_1}  =  \lim_{n\to \infty}  \tr K_n = \int_a^b q(t)R(t,t)\rd t
= \tr K.\quad\bl
$$

\

Let us return now to the Schr\"odinger operators $L_0$ and $L_1$ as in \rf{eq:def,of,Schr,oper-s}.

\begin{thm}
Let $q\in L^1(\Bbb R)\cap L^2(\Bbb R)$
and let $\im q \ge 0$. Then the pair of maximal dissipative operators
$(L_0,L_1)$ in $L^2(\Bbb R)$  has  a real-valued spectral shift function   $\bs{\xi}$ satisfying 
{\em\rf{skvadratom}}.
%Suppose that $\{L_0,L_1\}$ is a pair of dissipative operators satisfying the hypotheses of Theorem {\em\ref{imL0}}. Assume also that %$\trace(L_1-L_0)\not\in\R$.
%Then the pair $\{L_0,L_1\}$ cannot have an integrable real-valued spectral shift function.
\end{thm}

\Pf
We divide the proof in several steps.

(i) First, we observe that in accordance with the Sobolev imbedding theorem,
$W^{2,2}(\Bbb R)$ is continuously imbedded into the space $C_{\rm b}(\Bbb R)$  
of bounded continuous functions on $\R$ endowed with the sup-norm. Moreover, 
for each $\e > 0$,
there exists a constant  $b(\e) > 0$ such that
the following estimate holds:
   \begin{equation}\label{eq:q,is,str-ly,sub-d,to,D2}
 \|f\|_{C_{\rm b}(\Bbb R)}  \le \e \|f\|_{W^{2,2}} + b(\e)\|f\|_{L^2}  = \e \|D^2 f \|_{L^2} + b(\e)\|f\|_{L^2},
\quad f\in W^{2,2}(\Bbb R).
   \end{equation}

Therefore the assumption $q\in L^2(\Bbb R)$ implies that 
$
{\rm a}_{D^2}(M_q)=0
$
, i.e.,
\begin{equation}
     \label{eq:q,is,strongly,sub-d,to,Delta}
\|qf\|_{L^2} \le \|q\|_{L^2}\|f\|_{C_{\rm b}(\Bbb R)}  \le  \e_0 \|L_0 f \|_{L^2} + b(\e_0)\|f\|_{L^2},
\quad f\in \dom(L_0) = W^{2,2}(\Bbb R),
    \end{equation}
where $\e_0 = \e\|q\|_{L^2}$. This inequality means that the (dissipative) multiplication 
operator $M_q$
is dominated by $L_0$ with $L_0$-bound equal to 0, i.e., ${\rm a}_{L_0}(M_q)=0$.
Therefore, the Kato-Rellich theorem applies  and ensures that the operator  $L_1$ is closed on $\dom(L_0)$ and
    \begin{equation}\label{eq:dom,L,1=dom,L,0}
\dom(L_1) = \text\dom(L_1^*) = \dom(L_0)\ = W^{2,2}(\Bbb R).
    \end{equation}
In turn, this relation implies that  the operator $K\df (L_0 +\ri I) (L_1^* -\ri I)^{-1}$ is bounded.

%?????(ii) At this step we show that  $M_q$ is infinitesimally form subordinated
%to $L_0$ in $L^2(\Bbb R)$, i.e.,

(ii) Below we use the notation $\dom(\frak t)$ for the domain of a quadratic form $\frak t$.
At this step we consider the form $\frak t_{|q|}$,
$$
\frak t_{|q|}[f] = \big\|\sqrt{ |q|} f\big\|_{L^2}^2,  \quad
f\in \dom(\frak t_{|q|}) = \{f\in L^{2}(\R):\ \sqrt{ |q|} f \in L^{2}(\R) \}
$$
and the form $\frak t_{D^2}$,
$$
\frak t_{D^2}[f] = \|Df\|_{L^{2}(\R)}^2, \quad \dom(\frak t_{D^2}) = W^{1,2}(\Bbb R).
$$
Let us show that  for each $\varepsilon > 0$,
there exists a constant  $c(\varepsilon) > 0$ such that
the following estimate holds:
    \begin{equation}\label{eq:t_q,is,sub-d,to,t_D,first}
\frak t_{|q|}[f]  \le \e \frak t_{D^2}[f] + c(\e)\|f\|_{L^2}^2, \qquad f\in \dom(\frak t_{D^2}) = W^{1,2}(\Bbb R).
   \end{equation}

Now instead of estimate \eqref{eq:q,is,str-ly,sub-d,to,D2}, we need another
Sobolev imbedding theorem:
%%the imbedding $W^{2,2}(\Bbb R) \to C(\Bbb R)$  is continuous and
for each $\varepsilon > 0$,
there exists a constant  $c(\varepsilon) > 0$ such that
the following estimate holds:
  \begin{equation}\label{eq:t_q,is,str-ly,sub-d,to,t_D}
\|f\|_{C(\R)}^2  \le \varepsilon \|Df\|_{L^{2}(\Bbb R)}^2 + c(\varepsilon)\|f\|_{L^2(\Bbb R)}^2,
 %%= \varepsilon \|D^2 f \|_{L^2} + b(\varepsilon)\|f\|_{L^2},
\qquad f\in W^{1,2}(\Bbb R).
   \end{equation}
Using the assumption $q\in L^1(\Bbb R)$, applying estimate \eqref{eq:t_q,is,str-ly,sub-d,to,t_D},
and setting $C_q = \|q\|_{L^1(\R)}$,  we obtain
\begin{align*}
\frak t_{|q|}[f] &= \| q f^2\|_{L^1(\Bbb R)} \le \|q\|_{L^1(\Bbb R)}\|f\|_{C_{\rm b}(\Bbb R)}^2\\[.2cm]
&\le   C_q \left(\e \|Df\|_{L^{2}(\Bbb R)}^2 + c(\e)\|f\|_{L^2(\Bbb R)}^2\right), \quad f\in W^{1,2}(\Bbb R),
\end{align*}
This estimate shows that $\dom(\frak t_{|q|}) \supset \dom(\frak t_{D^2}) = W^{1,2}(\Bbb R)$ and
 proves estimate \eqref{eq:t_q,is,sub-d,to,t_D,first}.

By the KLMN theorem (\cite[Theorem X.17]{RS}), the form $\frak t_{D^2} + \frak t_{|q|}$ is semibounded below and closed on the domain
$\dom(\frak t_{D^2} + \frak t_{|q|}) = \dom(\frak t_{D^2}) = W^{1,2}(\Bbb R)$.
Moreover, the operator associated  with the form $\frak t_{D^2} + \frak t_{|q|}$  is ${D^2 + |q|I}$,  
 $\dom({D^2 + |q|I}) = W^{2,2}(\Bbb R)$.

It follows that both  operators $M_{\sqrt{ |q|}}(L_0 + \ri I)^{-1/2}$  and $M_{\sqrt{ |q|}}(L_1 +\ri I)^{-1/2}$ are bounded.

(iii)  It is easily seen (and well known) that the resolvent $(D^2 - zI)^{-1}$  of the selfadjoint operator
$D^2$ in $L^2(\Bbb R)$ admits the following integral representation
   \begin{equation}\label{eq:resol-t_of_H_on_line}
(D^2 - zI)^{-1}f = \int^{\infty}_{-\infty}G(x,t;z)f(t)\rd t, \quad f\in L^2(\Bbb R),
  \end{equation}
where  $G(x,t;z)$ is the Green's function of the Hamiltonian $D^2$,
\begin{equation}\label{eq:Green_for-la}
G(x,t;z)=
-\frac{1}{2\ri \sqrt z}\begin{cases}
e^{\ri(x-t)\sqrt z}, & \, \text{if f }t\le x,   \\
e^{-\ri(x-t)\sqrt z}, & \,  \text{if }t\ge x.
\end{cases}
\end{equation}
the function $G(x,t;z)$ is the Green's function of the Hamiltonian $D^2$.

Let us show that 
$M_q(D^2 - zI)^{-1}\in \bS_1$. To this end, we show  that the operator $R$ defined by
%$$
%R(-1) \df M_{|q|^{1/2}}(D^2 + I)^{-1}M_{|q|^{1/2}},
%$$
   \begin{align}
   \label{eq:resol-t_of_H_on_R,z=-1}
Rf &=  |q|^{1/2}(D^2 + I)^{-1}(|q|^{1/2} f)\nonumber\\[.2cm]
&=
\int^{\infty}_{-\infty} |q(x)|^{1/2}G(x,t;-1)|q(t)|^{1/2}f(t)\rd t, \quad f\in L^2(\Bbb R),
  \end{align}
belongs to $\bS_1$. First, we observe that for each  $t>0$, the (bounded) operator $(D^2+tI)^{-1}$
is  nonnegative  because so is
$D^2$. Therefore,  the operator $R$ is also nonnegative.

Next,  it follows from \eqref{eq:Green_for-la}  that
\begin{equation}\label{eq:Green_f-la,z=-1}
G(x,t;-1)=
\frac{1}{2}\begin{cases}
%\displaystyle
e^{-(x-t)}, & \, \text{if }t\le x,    \\
%{%\displaystyle
e^{(x-t)}, & \,  \text{if }t\ge x.
\end{cases}
\end{equation}
and so, $G(x,x;-1) = \frac{1}{2}$. Thus,  denoting by 
$k_R$ the kernel of the integral operator $R$ defined by
\eqref{eq:resol-t_of_H_on_R,z=-1} and keeping the assumption $q\in L^1(\Bbb R)$ in mind,
we can easily obtain 
\begin{equation}\label{eq:trace,ineq-ty}
\int^{\infty}_{-\infty}k_R(x,x)\,\rd x = \int^{\infty}_{-\infty} |q(x)|^{1/2}G(x,x;-1)|q(x)|^{1/2}\,\rd x
= \frac{1}{2}\int^{\infty}_{-\infty} |q(x)|\,\rd x < \infty.
   \end{equation}
Since the operator $R$ is  nonnegative, inequality \eqref{eq:trace,ineq-ty}  implies by Theorem \ref{Kuzya}
the required membership  of $R= M_{|q|^{1/2}}(D^2 + I)^{-1}M_{|q|^{1/2}}$ in $\bS_1$.

Similarly, we get from \eqref{eq:Green_f-la,z=-1} combined with the assumption   $q\in L^1(\Bbb R)$
that the operator $M_{|q|^{1/2}}(D^2 + I)^{-1}$ belong to the Hilbert-Schmidt
class $\bS_2$.

(iv)  
Next, since $\im q\ge 0$, we can  select a branch of the square  root of $q$
such that  $\re \sqrt {q(x)}>0$. Therefore,   
$$
R_0 \df M_{q^{1/2}}(L_0 + \ri I)^{-1} = M_{\vk|q|^{1/2}}(L_0 + i I)^{-1} \in \bS_2,
$$
 where $\vk(x) = q^{1/2}(x)/ |q^{1/2}(x)|$,\  $|\vk(x)|=1$.
 
 Next, we have the factorization
  \begin{equation}\label{eq:factor-n,of,res,dif-ce,new}
(L_1 + \ri I)^{-1}-(L_0 + \ri I)^{-1} = - \left(M_{q^{1/2}}(L_1^* -  \ri I)^{-1}\right)^* \left( M_{q^{1/2}}(L_0 + \ri I)^{-1}\right),
    \end{equation}
with both factors bounded.

Keeping in mind that the operator $K =\left((L_0 + \ri I)(L_1^* -  \ri I)^{-1}\right)$ is bounded (step (i)), we obtain
\begin{align*}
R_1 \df M_{q^{1/2}}(L_1^* -  \ri I)^{-1} 
&= (M_{q^{1/2}}(L_0 + \ri I)^{-1})\left((L_0 + \ri I)(L_1^* -  \ri I)^{-1}\right)\\[.2cm]
&= (M_{q^{1/2}}(L_0 + \ri I)^{-1})K \in \bS_2.
%%= R_0\left((L_0 + i I)(L_1^* -  i I)^{-1}\right)
\end{align*}

Combining both inclusions $R_j\in \bS_2$, $j=0,1$,  with identity \eqref{eq:factor-n,of,res,dif-ce,new}, we obtain
    \begin{equation}\label{eq:4.3res,dif-ce,is,tr,cl}
(L_1 + \ri I)^{-1}-(L_0 + \ri I)^{-1}= R_1^*R_0\in \bS_1.
    \end{equation}
 Thus, condition \eqref{eq:res,dif-ce,of,trace,cl} of Theorem \ref{imL0} is verified.

(v)
Next, it follows from  \eqref{eq:def,of,Schr,oper-s}  that
$$
\im L_0 = D^2  +  I \ge I  \qquad \text{ and} \qquad \im L_1= D^2  + I + \im q \ge I
$$
and due to  \eqref{eq:dom,L,1=dom,L,0}  $\text{dom} \im L_1 = \text{dom} \im L_0 $ and
since both operators $\im L_0, \  \im L_1$ are positive,  this implies that
$\text{dom} (\im L_1)^{1/2} = \text{dom} (\im L_0)^{1/2}$  (see e.g., \cite{RS}, Theorem X.18 and \cite{Sch}, Ex. 10.8.7).
In turn, this equality ensures  the existence  of a bounded  operator $K_1$
with bounded inverse  $K_1^{-1}$  that satisfies the condition
  $(\im L_1)^{-1/2}  = K_1  (\im L_0)^{-1/2}$.

On the other hand, it follows from \eqref{eq:def,of,Schr,oper-s} that $L_1 - L_0 = M_q$.  Inserting in \eqref{Im-1} this equality together with 
  the previous identity turns it into the following one
    \begin{equation}
    \label{Im-1,for,Schrod}
(\im L_1)^{-1/2}  M_q  (\im L_0)^{-1/2} =  K_1(\im L_0)^{-1/2}  M_q  (\im L_0)^{-1/2} \in \bS_1.
   \end{equation}
which we have to verify this  condition.
Since $K_1$ is invertible, 
%$K^{-1}\in \mathcal B(\frak H)$, 
condition \eqref{Im-1,for,Schrod}
is equivalent to the fact that
\begin{multline*}
\left(M_{|q|^{1/2}}(\im L_0)^{-1/2}\right)^* \left(M_{|q|^{1/2}}(\im L_0)^{-1/2}\right) \\[.2cm]
= \clos{\left((\im L_0)^{-1/2}M_{|q|^{1/2}}\right)} \left(M_{|q|^{1/2}}(\im L_0)^{-1/2}\right)   \in \bS_1,
\end{multline*}
where we use the notation $\clos A$ for the closure of a closable operator $A$.
In turn, this condition is equivalent to the condition
    \begin{equation}\label{Im-1,for,Schrod,2}
   R=\left(M_{|q|^{1/2}}(\im L_0)^{-1/2}\right)    \clos{\left((\im L_0)^{-1/2}M_{|q|^{1/2}}\right)}  \in \bS_1
   \end{equation}
  which follows immediately from the membership
 $R\in \bS_1$ established at step (iii). $\bl$

\

\

 \begin{footnotesize}

\noindent
\begin{tabular}{p{7cm}p{15cm}}
M.M. Malamud & V.V. Peller \\
Petersburg State University & St.Petersburg State University\\
 Universitetskaya nab., 7/9 &Universitetskaya nab., 7/9 \\
199034 St.Petersburg, Russia& 199034 St.Petersburg, Russia\\
email: malamud3m@gmail.com

\\
&St.Petersburg Department\\
&Steklov Institute of Mathematics\\
&Russian Academy of Sciences\\
&Fontanka 27, 191023 St.Petersburg\\
&Russia\\
& email: peller@math.msu.edu
\end{tabular}

\end{footnotesize}

\end{document}